\documentclass[12pt,letter]{article}
\usepackage{graphicx,colordvi,psfrag}
\usepackage{amsmath,amssymb}
\usepackage{amsthm}
\usepackage{epstopdf}
\usepackage{epsfig}
\usepackage{calc,pstricks, pgf, xcolor}
\usepackage{bbding,enumerate}
\usepackage{dsfont}
\usepackage{cite}

\title{An Upper Bound on the Sizes of Multiset-Union-Free Families}
\author{Or Ordentlich and Ofer Shayevitz \thanks{The authors are with the Department of Electrical Engineering - Systems at the Tel Aviv University, \{ordent,ofersha\}@eng.tau.ac.il. The work of O. Ordentlich was supported by the Admas Fellowship Program of the Israel Academy of Science and Humanities, a fellowship from The Yitzhak and Chaya Weinstein Research Institute for Signal Processing at Tel Aviv University, and the Feder Family Award. The work of O. Shayevitz was supported in part by the Marie Curie Career Integration Grant (CIG), Grant agreement no. 631983, and in part by the Israel Science Foundation under Grant No.  1367/14.}}

\def\wt{\widetilde}
\def\Expt{\mathbb{E}}
\def\b{\textbf}

\newcommand{\m}{\mathcal}

\newtheorem{lemma}{Lemma}
\newtheorem{proposition}{Proposition}
\newtheorem{theorem}{Theorem}
\newtheorem{remark}{Remark}
\newtheorem{corollary}{Corollary}

\begin{document}

\date{}
\maketitle

\thispagestyle{empty}
\begin{abstract}
Let $\m{F}_1$ and $\m{F}_2$ be two families of subsets of an $n$-element set. We say that $\m{F}_1$ and $\m{F}_2$ are multiset-union-free if for any $A,B\in \m{F}_1$ and $C,D\in \m{F}_2$ the multisets $A\uplus C$ and $B\uplus D$ are different, unless both $A = B$ and  $C= D$. We derive a new upper bound on the maximal sizes of multiset-union-free pairs, improving a result of Urbanke and Li.
\end{abstract}

\section{Introduction}\label{sec:intro}
Let $\m{F}_1$ and $\m{F}_2$ be two families of subsets of an $n$-element set. We say that $\m{F}_1$ and $\m{F}_2$ are \textit{multiset-union-free} if the multiset union of the families $\m{F}_1$ and $\m{F}_2$, defined as
\begin{align}
\m{F}_1\uplus\m{F}_2\triangleq\{F_1\uplus F_2 : F_1\in\m{F}_1,F_2\in\m{F}_2 \} \quad \text{with multiplicities}\nonumber
\end{align}
contains exactly $|\m{F}_1|\cdot|\m{F}_2|$ distinct elements. It would sometimes be instructive to represent $\m{F}_i$ by the corresponding set  $\m{C}_i$ of binary characteristic vectors; the multiset-union-free property is then equivalent to the requirement that $\b{a}+\b{c}\neq \b{b}+\b{d}$ for any vectors $\b{a},\b{b}\in \m{C}_1$ and $\b{c},\b{d}\in \m{C}_2$ unless both $\b{a}=\b{b}$ and $\b{c}=\b{d}$, where addition is over the reals. We say that a pair $0\leq R_1\leq 1$, $0\leq R_2\leq 1$ is admissible if there exists a sequence of multiset-union-free pairs $\m{F}_1$ and $\m{F}_2$ with cardinalities $|\m{F}_1| = 2^{n(R_1+o(1))}$ and $|\m{F}_2| = 2^{n(R_2+o(1))}$. Our goal is to find necessary conditions for a pair $(R_1,R_2)$ to be admissible. The set of all admissible $(R_1,R_2)$ has been extensively studied in the information theory literature; it is often referred to as the zero-error capacity region of the binary adder channel~\cite{Lindstrom69,Tilborg78,kl78,Weldon78,klwy83,bt85,bb98,UL98,ab99,mo05}.

Clearly, $R_1+R_2\leq \log{3} \approx 1.5849$ must hold, where logarithms are taken in base $2$. This bound can be easily improved via standard information theoretic arguments. Recall that the entropy of a random variable $X$ with a probability distribution $P=(p_1,\ldots,p_K)$ is
\begin{align*}
H(X) = -\sum_{k}p_k\log p_k.
\end{align*}
When convenient, we also denote the entropy of $X$ above by $H(P)$. Assume $\m{F}_1,\m{F}_2$ are multiset-union-free families with cardinalities $2^{nR_1}$ and $2^{nR_2}$ respectively. Let ${\bf X}_1, {\bf X}_2$ be two characteristic vectors pertaining to subsets in $\m{F}_1,\m{F}_2$ respectively chosen uniformly at random, independently of each other. The real sum ${\bf X}_1+ {\bf X}_2$ is hence uniformly distributed over all  $|\m{F}_1|\cdot |\m{F}_2| = 2^{n(R_1+R_2)}$ possible sums. By the subadditivity of entropy \cite{Cover}
\begin{align*}
  n(R_1+R_2) = H({\bf X}_1 + {\bf X}_2) \leq \sum_{k=1}^n H(X_{1,k}+X_{2,k}) \leq n\cdot \max_{P_{X_1},P_{X_2}} H(X_1+X_2)
\end{align*}
where the maximization is over all independent binary random variables $X_1,X_2$. The maximum is attained for uniform $P_{X_1}$ and $P_{X_2}$, which yields the bound $R_1+R_2 \leq \tfrac{3}{2}$.

Write $h(p) = H(p, 1-p)$ for the binary entropy, and $h^{-1}(x)$ for its inverse restricted to $[0,\tfrac{1}{2}]$. To date, the only improvement over the simple bound above was given by Urbanke and Li:
\begin{theorem}[Urbanke and Li~\cite{UL98}]
Any admissible $(R_1,R_2)$ satisfies
\begin{align*}
R_1+R_2\leq\min_{ 0\leq \rho\leq \tfrac{1}{2}}\max_{0\leq \kappa\leq 1}h\left(\langle 1-h^{-1}(R_1)-\kappa \rangle\right)-h(\rho)+\min\left\{g^*(\rho),\langle \rho+\kappa \rangle+h(\langle \rho+\kappa \rangle) \right\}
\end{align*}
where $\langle a\rangle\triangleq\min(a,1/2)$, and
\begin{align}
g^*(\rho)=\max_{0\leq\beta\leq 1}h\left(((1-\rho)(1-\beta), \;\rho(1-\beta)+(1-\rho)\beta, \;\rho\beta) \right).\nonumber
\end{align}
\label{thm:ul98}
\end{theorem}
For the maximal value of $R_1=1$, this bound yields $R_2<0.49216$, which improves upon $R_2\leq 0.5$ given by the standard information theoretic bound.

For $0\leq p,q \leq 1$, write $p\star q \triangleq p(1-q) + q(1-p)$. Let
  \begin{align}\label{eq:LJ}
\nonumber L(\eta) &\triangleq h(\eta) + 1- \eta \\
    J(p,\eta) &\triangleq \left\{\begin{array}{cc}
                                                                                             2h\left(\frac{1}{2}\left(1-\sqrt{1-2\eta}\right)\right)-\eta & \eta\geq p\star p \\
                                                                                             2h\left(\frac{1}{2}\left(1-\frac{1-\eta-p\star p}{\sqrt{1-2 (p\star p)}}\right)\right)-\frac{1}{2}\left(1-\frac{(1-\eta-p\star p)^2}{1-2 (p\star p)} \right) & \eta<p\star p
                                                                                           \end{array}
 \right.
  \end{align}
and
\begin{align}\label{eq:Rsum}
R_\Sigma(r_0,r_1) \triangleq \max_{h^{-1}(r_1)\leq \eta\leq \frac{1}{2}} \min\{L(\eta),\, J(h^{-1}(r_1),\eta)+r_0\}
\end{align}

Our main result is the following.
\begin{theorem}\label{thm:main}
Any admissible $(R_1,R_2)$ satisfies
\begin{align*}
  R_2 < \min_{0\leq \alpha \leq h^{-1}(R_1)} (1-\alpha) \left(R_\Sigma\left(\frac{\alpha}{1-\alpha},\,\Gamma(R_1,\alpha)\right) - \Gamma(R_1,\alpha)\right)
\end{align*}
where
\begin{align*}
  \Gamma(R_1,\alpha) \triangleq h\left(\frac{h^{-1}(R_1)-\alpha}{1-\alpha}\right)
\end{align*}
\end{theorem}
For the maximal value of $R_1=1$, this bound yields $R_2<0.4798$, which improves upon Theorem \ref{thm:ul98}. Figure \ref{fig1} depicts the three bounds for values of $R_1$ close to $1$.

The question of whether $R_1+R_2 = \tfrac{3}{2}$ is admissible for some $(R_1,R_2)$ remains wide open. We also note that there is a large gap between our bound and the best known constructions. For $R_1=1$, only $R_2=\tfrac{1}{4}$ is known to be admissible \cite{klwy83}, and the best known construction for the sum \cite{mo05} yields $R_1+R_2\approx 1.31781$.

\begin{figure}[htbp]\label{fig1}
  \centering
      \includegraphics[width=0.75\textwidth]{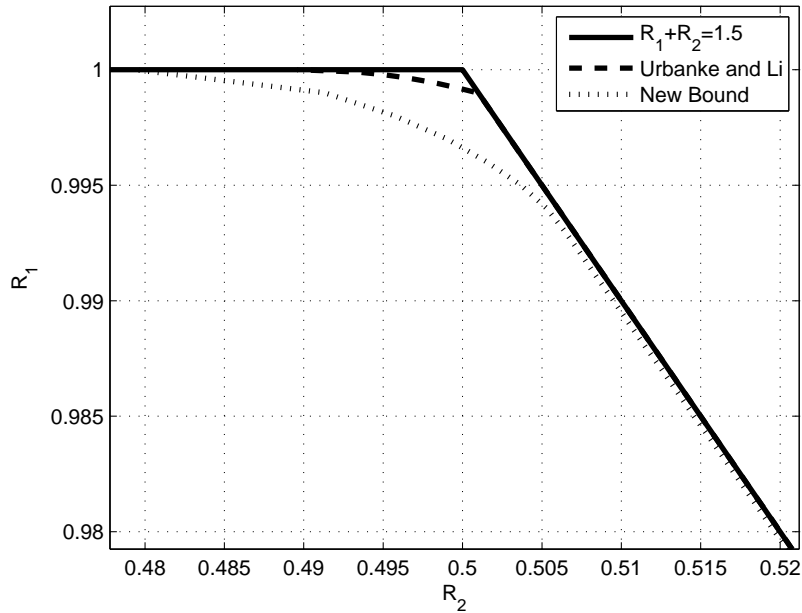}
  \caption{Illustration of the three bounds.}
\end{figure}

\section{Proof of Theorem \ref{thm:main}}

To avoid cumbersome notations, and since admissibility is an asymptotic property, we can assume without loss of generality that $nR_1$ and $nR_2$ (and all similar quantities) are integers.

\subsection{Motivation}
Let $\m{F}$ be a family of subsets of $[n] \triangleq\{1,\ldots,n\}$, and $S\subseteq [n]$. We say that $S$ is \textit{shattered} by $\m{F}$ \cite{AlonSpencer}, if the \textit{projection multiset} (or simply \textit{projection})
\begin{align*}
P_S^+(\m{F}) \triangleq \left\{F\cap S: F\in\m{F} \right\} \qquad \text{with multiplicities}
\end{align*}
of $\m{F}$ on $S$ contains all subsets of $S$.\footnote{Taking the multiplicities into account in the definition of the projection is not necessary here, but will become important in the sequel.} A family $\m{F}$ is said to be \textit{systematic} if it is shattered by some $S\subseteq [n]$ of cardinality $ \log |\m{F}|$. Weldon proved the following \cite{Weldon78}.
\begin{theorem}[Weldon~\cite{Weldon78}]\label{thm:weldon}
  If $\m{F}_1$ is systematic and the pair $\m{F}_1,\m{F}_2$ is multiset-union-free, then $R_2 \leq (1-R_1)\log{3}$.
\end{theorem}
\begin{proof}
  Let $S$ be a set of cardinality $  nR_1  $ that is shattered by $\m{F}_1$. For every $F_2\in\m{F}_2$, there exists an $F_1\in \m{F}_1$ such that $F_1$ and $F_2$ are an \textit{$S$-complement} pair, i.e.,
  \begin{align}\label{eq:all1}
(F_1\cap S)\uplus (F_2\cap S) = S.
  \end{align}
Hence, there are at least $2^{nR_2}$ such $S$-complement pairs. By the multiset-union-free assumption, $(F_1\cap\overline{S}) \uplus (F_2\cap\overline{S})$ must be distinct for all $S$-complement pairs. Therefore, the number of such pairs cannot be larger than $3^{|\overline{S}|} = 3^{n(1-R_1)}$, and the theorem follows.
\end{proof}
For example, if $\m{F}_1$ is systematic and $R_2=1$, then the theorem implies that $R_1 \leq 0.37$. This strong bound is a consequence of the restriction to a systematic family. However, we note that the only property used in the proof is the existence of a large shattered set. Hence, any lower bound on the size of a maximal shattered set in a general family $\m{F}_1$ would lead to a similar result. The Sauer-Perles-Shelah lemma provides such a guarantee.
\begin{lemma}[Sauer-Perles-Shelah \cite{AlonSpencer}]
  Let $\m{F}$ be a family of subsets on an $n$-element set. If the cardinality of the maximal subset shattered by $\m{F}$ is $d$, then $|\m{F}| \leq \sum_{k=0}^d{n \choose k}$.
\end{lemma}
\begin{remark}
  It is easy to see that this bound is attained with equality if $\m{F}$ is a $n$-Hamming ball of radius $d$.
\end{remark}
\begin{corollary}
Let $\varepsilon>0$. If $|\m{F}| = 2^{n(R+\varepsilon)}$ then for any $n$ large enough, $\m{F}$ shatters a set $S\subseteq [n]$ with $|S|\geq n h^{-1}(R)$.
\end{corollary}

Plugging the above into Weldon's argument yields:
\begin{proposition}
If the pair $\m{F}_1,\m{F}_2$ is multiset-union-free, then $R_2 \leq (1-h^{-1}(R_1))\log{3}$.
\end{proposition}
Unfortunately, this bound is trivial since $R_1+(1-h^{-1}(R_1))\log{3} > \tfrac{3}{2}$ for any $R_1$. This stems from two main weaknesses. First, we have taken the worse case assumption that each subset $F_2\in \m{F}_2$ has only one subset $F_1\in\m{F}_1$ such that $F_1$ and $F_2$ are $S$-complement, where $S$ is a shattered set in $\m{F}_1$. Second, bounding the number of $S$-complement pairs by $3^{|\overline{S}|}$ may be loose, as it ignores the multiset union structure. In the next two subsections, we provide the technical tools to handle each of these weaknesses. We then apply them to prove the theorem in the subsection that follows.

\subsection{A Soft  Sauer-Perles-Shelah Lemma}

Let $\m{F}$ be a family of subsets of $[n] \triangleq\{1,\ldots,n\}$, and $S\subseteq [n]$. We say that $S$ is \textit{$k$-shattered} by $\m{F}$, if the projection multiset
$P_S^+(\m{F})$ of $\m{F}$ on $S$ contains all subsets of $S$ each with multiplicity of at least $k$. For $k=1$, this definition reduces to the regular definition of a shattered set.

In Section \ref{sec:proof_soft_sauer}, we prove the following Lemma.
\begin{lemma}\label{lem:soft_sauer}
  Let $\m{F}$ be a family of subsets of an $n$-element set. If the cardinality of the maximal subset that is $k$-shattered by $\m{F}$ is $d-1$, then
  \begin{align*}
     |\m{F}|\leq \sum_{t=1}^{t^*}{n\choose t}  + {n\choose t^*}\sum_{t=t^*+1}^{n}\frac{{t^*\choose d}}{{t \choose d}}
  \end{align*}
where $t^*$ is the smallest integer $t$ satisfying ${n-d \choose t-d} \geq k$ if such an integer exists, and $t^*=n$ otherwise.
\end{lemma}
\begin{remark}
Note that if $k={n-d \choose t^*-d}$ for some $t^*$, then our bound is tight for a $n$-Hamming ball of radius $t^*$, up to multiplicative gap of $O(n/d)$. This coincides with the Sauer-Perles-Shelah Lemma for $k=1$ (and $t^*=d$), up to the aforementioned multiplicative factor. Since we are only interested in exponential behavior, no attempt has been made to reduce this gap.
\end{remark}

\begin{corollary}\label{cor:soft_sauer}
Let $\varepsilon>0$. If $|\m{F}| = 2^{n(R+\varepsilon)}$ then for any $0\leq \alpha \leq h^{-1}(R)$ and any $n$ large enough, there exists a set $S\subseteq[n]$ with $|S|\geq n\alpha$ that is $\;2^{n\beta}$-shattered by $\m{F}$, where
  \begin{align}\label{eq:beta}
    \beta = (1-\alpha)\cdot h\left(\frac{h^{-1}(R)-\alpha}{1-\alpha}\right)
\end{align}
\end{corollary}
\begin{proof}
Let $0\leq \alpha \leq h^{-1}(R)$ and assume to the contrary that the claim does not hold. Denote $t^*=\gamma_n n$, and write
\begin{align*}
\frac{1}{n}\log{n-d\choose t^*-d} &= \frac{n-d}{n} \left(h\left(\frac{t^*-d}{n-d}\right) + o(1)\right) \\
&= (1-\alpha + o(1))h\left( \frac{\gamma_n - \alpha + o(1)}{1-\alpha + o(1)}\right)
\end{align*}
We can set $\gamma_n$ to the minimal value guaranteeing that the above is at least $\beta$, which is $\gamma_n =\alpha+(1-\alpha)h^{-1}\left(\frac{\beta}{1-\alpha}\right) + o(1)$. Invoking Lemma \ref{lem:soft_sauer}, it must then be that $|\m{F}|>  2^{n(h(\gamma_n) + o(1))} = 2^{n(R+o(1))}$, contradicting the assumption.
\end{proof}

\subsection{An Information Theoretic Lemma}
We define a natural generalization of the multiset-union-free property for sets of family pairs. A \textit{system} $\m{U}$ is a set of pairs $\{\m{F}_{1,i},\m{F}_{2,i}\}_{i=1}^{M_0}$, where each $\m{F}_{1,i}$ (resp.  $\m{F}_{2,i}$) is a family of subsets of $[n]$ with fixed cardinality $|\m{F}_{1,i}| = M_1$ (resp. $|\m{F}_{2,i}| = M_2$). We say that $\m{U}$ is a \textit{multiset-union-free system} if each pair $(\m{F}_{1,i},\m{F}_{2,i})$ is multiset-union-free, and the families of multisets $\m{F}_{1,i}\uplus \m{F}_{2,i}$ are mutually disjoint.

A triplet $(r_0,r_1,r_2)$ is called admissible if there exists a sequence of multiset-union-free systems $\m{U}$ with $M_\ell = 2^{n(r_\ell+o(1))}$ for $\ell\in\{0,1,2\}$. The goal of this subsection is to provide a necessary condition for a triplet to be admissible. In the Weldon-type arguments mentioned above, the number of $S$-complement pairs was bounded by $3^{|\overline{S}|}$, thereby ignoring the multiset union structure. As we shall see in the next subsection, this structure can be accounted for by partitioning each family according to its projection on $S$, which naturally gives rise to a system with $r_0\leq |S|/|\overline{S}|$. Moreover, any upper bound on the corresponding admissible sum $r_0+r_1+r_2$ can be translated into an upper bound on the number of $S$-complement pairs in our original setup.

For $r_0=0$, the problem coincides with the standard multiset-union-free problem, for which $r_0+r_1+r_2\leq \tfrac{3}{2}$ follows from the information theoretic argument given in Section \ref{sec:intro}. It is also easy to see that for a large enough value of $r_0$, the sum $r_0+r_1+r_2 = \log{3}$ is admissible. For example, let $\m{F}_0 = \{F_{0,1},\ldots,F_{0,M_0}\}$ be the set of all subsets of $[n]$ with cardinality $2n/3$, and identify each pair $\{\m{F}_{1,i},\m{F}_{2,i}\}$ in the system $\m{U}$ with one of the these subsets. Let $\m{F}_{1,i} = \{F_{0,i}\}$ and $\m{F}_{2,i}=\{F\subset [n]: F\subseteq F_{0,i}\}$. Clearly, each pair $(\m{F}_{1,i},\m{F}_{2,i})$ is multiset-union-free, and moreover, the families of multisets $\m{F}_{1,i}\uplus \m{F}_{2,i}$ as defined above are disjoint, as exactly all the elements of $F_{0,i}$ participate in each corresponding family of multisets. For this construction, $r_0=\tfrac{1}{n}\log {n \choose 2n/3} \approx h(\tfrac{1}{3})$, $r_1=0$ and $r_2 = \tfrac{2}{3}$, hence in the limit of large $n$ this construction yields $r_0+r_1+r_2 = \log{3}$. The next lemma refines these observations by upper bounding admissible sums $r_0+r_1+r_2$ between $\tfrac{3}{2}$ and $\log{3}$, as a function of $r_0$ and $r_1$. The proof appears in Section \ref{sec:proof_sw}.

\begin{lemma}\label{lem:sw}
  Let $L(\eta)$ and $J(p,\eta)$ be as defined in~\eqref{eq:LJ}. If $(r_0,r_1,r_2)$ is admissible, then
\begin{align*}
  r_0+r_1+r_2 \leq \max_{h^{-1}(r_1)\leq \eta\leq \frac{1}{2}} \min\{L(\eta),\, J(h^{-1}(r_1),\eta)+r_0\}
\end{align*}
\end{lemma}

\begin{remark}
  Note that it can be shown that the maximization can be further restricted to $h^{-1}(r_1)\star h^{-1}(r_2)\leq \eta \leq \frac{1}{2}$. This however is not useful for our purposes.
\end{remark}

\subsection{Putting it Together}
We are now in a position to prove Theorem \ref{thm:main}. Let $\m{F}_1,\m{F}_2$ be a pair of multiset-union-free families of cardinalities $2^{nR_1}$ and $2^{nR_2}$ respectively. Given this pair, we use Corollary \ref{cor:soft_sauer} to  construct a multiset-union-free system with certain cardinalities, and then apply Lemma \ref{lem:sw} to obtain constraints on that system.

By Corollary \ref{cor:soft_sauer}, for any $\alpha<h^{-1}(R_1)$ there exists a subset $S\subset [n]$ of cardinality $n\alpha$ that is $2^{n\beta}$-shattered by $\m{F}_1$, where $\beta$ is given in~\eqref{eq:beta}, all up to an $o(1)$ term. Let $\m{F}_0$ be the family of all subsets of $S$, and for any $G\in\m{F}_0$ let $\m{F}_{1,G} = \{F\in\m{F}_1 : F\cap S = G\}$. Define $\m{F}_{2,G}$ similarly, and note that $\{\m{F}_{i,G}\}_{G\in\m{F}_0}$ is a partition of $\m{F}_i$ for each $i\in\{1,2\}$.

By construction, $|\m{F}_{1,G}| \geq 2^{n\beta}$. We can therefore arbitrarily choose $\wt{\m{F}}_{1,G}\subseteq \m{F}_{1,G}$ such that $|\wt{\m{F}}_{1,G}| = 2^{n\beta}$. For each $G$ with $|\m{F}_{2,G}|>0$, arbitrarily choose $\wt{\m{F}}_{2,G}\subseteq \m{F}_{2,G}$ such that $\log|\wt{\m{F}}_{2,G}| = \lfloor\log|\m{F}_{2,G}|\rfloor$. Note that this guarantees that $|\wt{\m{F}}_{2,G}| = 2^k$ for some integer $0\leq k\leq nR_2$, and that $|\wt{\m{F}}_{2,G}| \geq |\m{F}_{2,G}|/2$. Moreover, there must exist an integer $k'$ with the property that the union of all $\wt{\m{F}}_{2,G}$ of cardinality $2^{k'}$ contains at least $\tfrac{1}{2(nR_2+1)}2^{nR_2}$ subsets. Let $\m{G}$ be the set of all $G\in\m{F}_0$ that correspond to this $k'$, and note that by construction $|\m{G}|= 2^{n\alpha'}$ for some $\alpha'\leq \alpha$. Moreover, $|\wt{\m{F}}_{2,G}| = 2^{k'} \geq  \tfrac{1}{2(nR_2+1)}2^{n(R_2-\alpha')}$ for all $G\in\m{G}$.

Let $\overline{G} = S\setminus G$, and define the system $\m{U} = \{(\wt{\m{F}}_{1,\overline{G}},\wt{\m{F}}_{2,G})\}_{G\in\m{G}}$. Since the original $\m{F}_1$ and $\m{F}_2$ are multiset-union-free, then $\m{U}$ is trivially a multiset-union-free system. Moreover, since any $F_1\in \wt{\m{F}}_{1,\overline{G}}$ and $F_2\in \wt{\m{F}}_{2,G}$ are an $S$-complement pair~\eqref{eq:all1}, the projection\footnote{Note that $P^+_{\overline{S}}(\wt{\m{F}}_{1,\bar{G}})$ and $P^+_{\overline{S}}(\wt{\m{F}}_{2,G})$ have no multiplicities.} $\m{U}_{\overline{S}} =  \{(P^+_{\overline{S}}(\wt{\m{F}}_{1,\overline{G}}), P^+_{\overline{S}}(\wt{\m{F}}_{2,G}))\}_{G\in\m{G}}$ of $\m{U}$ onto $\overline{S}$ is also a multiset-union-free system, over $|\overline{S}| = n(1-\alpha)$ elements.

We have thus shown that given a multiset-union-free pair over $[n]$ with cardinalities $2^{nR_1}$ and $2^{nR_2}$, we can construct a multiset-union-free system $\m{U}_{\overline{S}}$ over $[m]= [n(1-\alpha)]$ with cardinalities $M_0 =  2^{mr_0}$, $M_1 = 2^{mr_1}$ and $M_2 = 2^{m\left(r_2+o(1)\right)}$, where $r_0=\frac{\alpha'}{1-\alpha}$, $r_1=\frac{\beta}{1-\alpha}$ and $r_2=\frac{R_2-\alpha'}{1-\alpha}$. Thus for this system $r_0+r_1+r_2 = \frac{R_2+\beta}{1-\alpha}$, and by Lemma \ref{lem:sw} we have that
\begin{align*}
  \frac{R_2+\beta}{1-\alpha} &\leq \max_{h^{-1}\left(\frac{\beta}{1-\alpha}\right)\leq \eta\leq \frac{1}{2}} \min\left\{L(\eta),\, J\left(h^{-1}\left(\frac{\beta}{1-\alpha}\right),\eta\right)+\frac{\alpha'}{1-\alpha}\right\},
\end{align*}
where we have used $\alpha\leq \alpha$. The theorem now follows by substituting $\beta$ from Corollary \ref{cor:soft_sauer}, and noting that the inequality above holds for any $0\leq \alpha \leq h^{-1}(R_1)$.

\section{Proof of Lemma \ref{lem:soft_sauer}}\label{sec:proof_soft_sauer}

Let $\m{F}$ be a family of subsets on $[n]$. We start by applying the shifting argument introduced in \cite{Alon83}, to construct another family $\m{G}$ of the same cardinality, such that if $S$ is $k$-shattered by $\m{G}$ then it is also $k$-shattered by $\m{F}$. Furthermore, $\m{G}$ will be \textit{monotone}, i.e., will have the property that if $G\in\m{G}$ then all subsets of $G$ are in $\m{G}$.

Set $\m{G}=\m{F}$. If $\m{G}$ is already monotone, we are done. Otherwise there exists some $i\in[n]$ such that  the set $\wt{\m{G}}_i = \{G\in\m{G} : i\in G,\,G\setminus \{i\}\not\in\m{G}\}$ is not empty. Update $\m{G}$ according to the rule:
\begin{align}\label{eq:update}
\m{G} \leftarrow \left(\m{G}\setminus \wt{\m{G}}_i\right) \cup \left(\wt{\m{G}}_i - i\right)
\end{align}
where $\wt{\m{G}}_i - i$ is the family of subsets obtained from $\wt{\m{G}}_i$ by removing the element $i$ from each subset. The process continues until $\m{G}$ is monotone, and is clearly guaranteed to terminate in finite time. By construction, $|\m{G}|=|\m{F}|$.

We now show that if $S$ is $k$-shattered by $\m{G}$ then it is also $k$-shattered by $\m{F}$. Let $\m{G}'$ be the family of subsets before the operation~\eqref{eq:update} on some element $i$, and let $\m{G}$ be the family obtained after that operation. Suppose $S$ is $k$-shattered by $\m{G}$. It now suffices to show that $S$ is also $k$-shattered by $\m{G}'$. If $i\not\in S$ then clearly $P_S^+(\m{G}) = P_S^+(\m{G}')$, hence this does not affect the $k$-shatterdness of $S$. Suppose $i\in S$, and let  $\m{G}_i  =\{G\in\m{G} : i\in G\}$. Then $\m{G}_i\subseteq \m{G}'$ since the update rule~\eqref{eq:update} does not add elements to subsets. Since $\m{G}$ $k$-shatters $S$, then every subset of $S$ that contains $i$ has multiplicity at least $k$ in $P_S^+(\m{G}_i)\subseteq P_S^+(\m{G}')$. Recalling that $\m{G}_i\subseteq \m{G}\cap\m{G}'$, we have that $\m{G}_i - i \subseteq \m{G}'$ since otherwise some replacement would have occurred in~\eqref{eq:update}. Since $\m{G}$ $k$-shatters $S$, then every subset of $S$ that does not contain $i$ has multiplicity at least $k$ in $P_S^+(\m{G}_i - i)\subseteq P_S^+(\m{G}')$.


The Lemma now follows directly from the next proposition.

\begin{proposition}
  If $\m{G}$ is a monotone family of subsets of $[n]$ with the property that no subset of cardinality $d$ is $k$-shattered by $\m{G}$, then
  \begin{align*}
     |\m{G}|\leq \sum_{t=1}^{t^*}{n\choose t}  + {n\choose t^*}\sum_{t=t^*+1}^{n}\frac{{t^*\choose d}}{{t \choose d}}
  \end{align*}
where $t^*$ is the smallest integer $t$ satisfying ${n-d \choose t-d} \geq k$ if such an integer exists, and $t^*=n$ otherwise.
\end{proposition}

\begin{proof}
Let $\m{G}_t$ denote the family of all subsets in $\m{G}$ with cardinality $t$. For $t\geq d$, every $G\in\m{G}_t$ has exactly ${t \choose d}$ subsets of cardinality $d$. There is a total of ${n \choose d}$ subsets of cardinality $d$. Hence by a simple counting argument there must exist at least one subset $S$ of cardinality $d$, that is a subset of no less than $|\m{G}_t|{t\choose d}/{n \choose d}$ subsets in $\m{G}_t$. Recalling that $\m{G}$ is monotone, this implies that $S$ is $|\m{G}_t|{t\choose d}/{n \choose d}$-shattered by $\m{G}$. By our assumption, it must be that
\begin{align*}
  \frac{{t\choose d}|\m{G}_t|}{{n \choose d}} < k,\quad t=d,\ldots,n
\end{align*}
On the other hand, $|\m{G}_t|\leq {n \choose t}$, and therefore
\begin{align*}
  |\m{G}_t| \leq \min\left\{{n \choose t}, \frac{{n\choose d}k}{{t \choose d}}\right\}, \quad t=d,\ldots,n
\end{align*}
Summing over $t$ we get
\begin{align}\label{eq:Gbound}
  |\m{G}| &= \sum_{t=1}^n|\m{G}_t| \leq  \sum_{t=1}^{d-1}{n\choose t} + \sum_{t=d}^n\min\left\{{n \choose t}, \frac{{n\choose d}k}{{t \choose d}}\right\}
\end{align}
Let $t^*$ be the smallest integer $t$ such that ${n \choose t} \geq \frac{{n\choose d}k}{{t \choose d}}$ if such an integer exists. If no such integer $t$ exists, set $t^* = n$. Then
\begin{align*}
  |\m{G}| \leq \sum_{t=1}^{t^*}{n\choose t}  + \sum_{t=t^*+1}^{n}\frac{{n\choose d}k}{{t^*\choose d}}\cdot\frac{{t^*\choose d}}{{t \choose d}} \leq \sum_{t=1}^{t^*}{n\choose t}  + {n\choose t^*}\sum_{t=t^*+1}^{n}\frac{{t^*\choose d}}{{t \choose d}} \end{align*}

To complete the proof, note that for any $d\leq t\leq n$ we have ${n\choose t}{t\choose d} = {n\choose d}{n-d\choose t-d}$, hence $t^*$ is the smallest integer $t$ satisfying ${n-d\choose t-d} \geq k$ if such an integer exists, and otherwise $t^*=n$.
\end{proof}

\section{Proof of Lemma \ref{lem:sw}}\label{sec:proof_sw}
We will need the following basic definitions and properties of entropy \cite{Cover}. The entropy of $X\sim\textrm{Uniform}([m])$ is $H(X) = \log{m}$. If $P = (p_0,p_1,\ldots,p_k)$ then the grouping rule for entropy states that $H(P) = H(p_0,\ldots,p_{k-2},p_{k-1}+p_k) + (p_{k-1}+p_k)h\left(\frac{p_{k-1}}{p_{k-1}+p_k}\right)$. In particular, if $P = (p_0,p_1,p_2)$, this implies that $H(P)\leq h(p_0) + 1-p_0$ with equality if and only if $p_1=p_2$. For two jointly distributed random variables $X,Y$, let $H(X|Y=y)$ denote the entropy of the distribution $P_{X|Y=y}$, and let $H(X|Y)$ be its expectation w.r.t. $P_Y$. The chain rule for entropies states that $H(X,Y) = H(Y) +H(X|Y)$. In addition, $H(X|Y) \leq H(X)$, i.e., conditioning reduces entropy. The latter two properties imply the sub-additivity of entropy, i.e., $H(X,Y) \leq H(X)+H(Y)$. Finally, note that the binary entropy function $h(\cdot)$ is symmetric around $\tfrac{1}{2}$.

Let $\m{V} = \{\m{F}_{1,i},\m{F}_{2,i}\}_{i=1}^{2^{nr_0}}$ be a \textit{multiset-union-free system}, where each $\m{F}_{1,i}$ (resp.  $\m{F}_{2,i}$) is a family of subsets of $[n]$ with fixed cardinality $|\m{F}_{1,i}| = 2^{nr_1}$ (resp. $|\m{F}_{2,i}| = 2^{nr_2}$). Let $V\sim\textrm{Uniform}([2^{nr_0}])$ be an index in the system, chosen uniformly at random. Let $\b{X}_1\sim \textrm{Uniform}(\m{C}_{1,V})$  and $\b{X}_2\sim \textrm{Uniform}(\m{C}_{2,V})$, where $\m{C}_{j,V}$ is the set of characteristic vectors corresponding to $\m{F}_{j,V}$. Note that this construction induces a joint distribution  $P_{V,{\bf X}_1,{\bf X}_2} = P_VP_{{\bf X}_1|V}P_{{\bf X}_2|V}$. Let $Q\sim \textrm{Uniform}([n])$ be a random coordinate of the characteristic vectors, mutually  independent of $({\bf X}_1, {\bf X}_2 ,V)$, and define the binary random variables $X_1 = X_{1,Q}$ and $X_2 = X_{2,Q}$.

By the multiset-union-free assumption, we have that $\b{X}_1+\b{X}_2$ is uniformly distributed over a set of cardinality $2^{n(r_0+r_1+r_2)}$. Using that and the sub-additivity of entropy, we have that
\begin{align}\label{eq:r012}
\nonumber  n(r_0+r_1+r_2) &= H({\bf X}_1+{\bf X}_2) \leq \sum_{q=1}^nH(X_{1,q}+X_{2,q})
= n\mathbb{E}H(X_{1,Q}+X_{2,Q})\\ &= nH(X_1+X_2|Q) \leq nH(X_1+X_2)
\end{align}
where the last inequality follows since conditioning reduces entropy. Similarly, we have that $n(r_1+r_2) = H({\bf X}_1+{ \bf X}_2|V=v)$ for any $V=v$, and hence
\begin{align}\label{eq:r12}
\nonumber  n(r_1+r_2) &= 2^{-nr_0}\sum_{v=1}^{2^{nr_0}}H({\bf X}_1+{\bf X}_2|V=v) = H({\bf X}_1+{\bf X}_2|V) \\ &\leq \sum_{q=1}^{n}H(X_{1,q}+X_{2,q}|V) = nH(X_2+X_1|V,Q)
\end{align}

Finally, we also have that $nr_1 = H({\bf X}_1|V=v)$ for any $V=v$ and hence
\begin{align}\label{eq:r1}
\nonumber  nr_1 &= 2^{-nr_0}\sum_{v=1}^{2^{nr_0}}H({\bf X}_1|V=v) = H({\bf X}_1|V) \\ &\leq \sum_{q=1}^{n}H(X_{1,q}|V) = nH(X_1|V,Q).
\end{align}

Combining~\eqref{eq:r012},~\eqref{eq:r12} and~\eqref{eq:r1}, and defining $U = (V,Q)$, we obtain the following.
\begin{proposition}\label{prop:SW_region}
If $(r_0,r_1,r_2)$ is admissible, then there exists $U\sim P_U$ of finite cardinality, and conditional binary distributions $P_{X_1|U}$ and $P_{X_2|U}$, such that
\begin{align}
\nonumber  r_0+r_1+r_2 &\leq H(X_1+X_2) \\
\nonumber r_1+r_2 &\leq H(X_1+X_2|U) \\
r_1 &\leq H(X_1|U)\label{eq:region}
\end{align}
where $P_{U,X_1,X_2} = P_UP_{X_1|U}P_{X_2|U}$.
\end{proposition}

\begin{remark}
  The above proposition is a special case of a general result of Slepian and Wolf \cite{SW73}.
\end{remark}

Following the proposition above, characterizing the set of all possible entropy triplets $(H(X_1+X_2),H(X_1+X_2|U),H(X_1|U))$ will result in necessary conditions for admissibility of triplets $(r_0,r_1,r_2)$. More precisely, it is our goal to characterize the set of \textit{extremal entropy triplets}, namely those entropy triplets that are Pareto optimal. We refer to the distributions $P_U,P_{X_1|U},P_{X_2|U}$ that achieve these extremal entropy triplets as \textit{extremal distributions}.

\begin{remark}
Using the Carath\'{e}odory's Theorem based technique initiated in \cite{AK75,WZ76,W82}, it can be shown that it suffices to consider $U$ of cardinality at most $3$. While this significantly reduces the dimension of the space of extremal distributions, the remaining number of parameters still renders a brute-force search prohibitive. Instead, in what follows we bound the extremal entropy triplets analytically.
\end{remark}

First, note that choosing $U=\emptyset$ and $X_1,X_2$ uniformly random, yields $H(X_1|U)=1$ and $H(X_1+X_2|U) = H(X_1+X_2) = \tfrac{3}{2}$, with $P_{X_1+X_2}(1) = \tfrac{1}{2}$. By the grouping property of entropy, if $P_{X_1+X_2}(1) > \tfrac{1}{2}$ then $H(X_1+X_2|U)\leq H(X_1+X_2) <  \tfrac{3}{2}$, hence any extremal distribution must satisfy $P_{X_1+X_2}(1) \leq \tfrac{1}{2}$. Furthermore, we show the following.
\begin{lemma}
Any extremal entropy triplet can be achieved by an extremal distribution inducing a $P_{X_1,X_2}$ that can be described by
\begin{align}\label{eq:dsbs}
  X_1\sim\textrm{Bern}(\tfrac{1}{2}), \; X_2 = X_1\oplus Z, \; Z \sim\textrm{Bern}(\eta)
\end{align}
for some $\eta\in[0,\tfrac{1}{2}]$, where $X_1$ and $Z$ are independent.
\end{lemma}
\begin{proof}
Consider any choice of $P_U$, $P_{X_1|U}$ and $P_{X_2|U}$, and without loss of generality assume the support of $U$ is the set $[m]$, for some finite $m$. We write $t_u,q_u$ for the Bernoulli parameters of $X_1|U=u$ and $X_2|U=u$ respectively. We construct another distribution satisfying~\eqref{eq:dsbs} that keeps the conditional entropies constant while not decreasing $H(X_1+X_2)$.

Define an extended distribution $W$ with support $[m]\cup (-[m])$, such that $P_W(w) = \tfrac{1}{2}P_U(|w|)$. Define further $\wt{t}_w = t_w$ for $w>0$ and $\wt{t}_w = 1-t_w$ otherwise. Let $\wt{q}_w$ be defined similarly. With some abuse of notation, let $P_{X_1|W}$ and $P_{X_2|W}$ follow the Bernoulli parameters $\wt{t}_w$ and $\wt{q}_w$ respectively. We will now refer to $X_1,X_2$ under $U$ or under $W$ to mean the obvious. Note that $P_{X_1|W=w}$ and $P_{X_1|W=-w}$ are identical up to substituting the probabilities of $0$ and $1$. Similarly, $P_{X_1+X_2|W=w}$ and $P_{X_1+X_2|W=-w}$ are identical up to substituting the probability assigned to $0$ and $2$. Hence, we clearly have that $H(X_1|W)=H(X_1|U)$ and $H(X_1+X_2|W) = H(X_1+X_2|U)$. For the same reason, $P_{X_1+X_2}(1)$ under $U$ and $P_{X_1+X_2}(1)$ under $W$ are the same. Furthermore, under $W$ we have that $P_{X_1+X_2}(0) = P_{X_1+X_2}(2)$, and so by the grouping rule for entropy we conclude that $H(X_1+X_2)$ under $W$ is not smaller than $H(X_1+X_2)$ under $U$. Moreover, from symmetry we have that $P_{X_1,X_2}(0,1) = P_{X_1,X_2}(1,0)$ under $W$. We can therefore think of $X_1,X_2$ under $W$ as being generated by~\eqref{eq:dsbs} for $\eta = \Pr(X_1\neq X_2) = P_{X_1+X_2}(1)$.

\end{proof}

We now restrict our attention to distributions of the form~\eqref{eq:dsbs}. Fix some $\eta$, and note that
\begin{align}\label{eq:hx1x2}
H(X_1+X_2) = h(\eta) + 1-\eta = L(\eta).
\end{align}
Our goal is therefore to maximize $H(X_1+X_2|U)$ subject to $H(X_1|U)\geq r_1$, over all $P_U,P_{X_1|U},P_{X_2|U}$ for which $P_{X_1,X_2}$ is consistent with~\eqref{eq:dsbs} and our $\eta$.

Define
\begin{align*}
a_u &= \Pr(X_1=0|U=u) \\ b_u &=\Pr(X_2=0|U=u)
\end{align*}
and the random variables $a\triangleq a_U$ and  $b\triangleq b_U$. Note that by definition
\begin{align}\label{eq:ent_ab}
  H(X_1|U) = \Expt h(a), \; H(X_2|U) = \Expt h(b)
\end{align}
Clearly
\begin{align*}
  H(X_1+X_2|U=u) = H\big{(}a_ub_u, (1-a_u)(1-b_u), a_u\star b_u\big{)}
\end{align*}
Moreover, by the grouping rule for entropy we can also write
\begin{align*}
  H(X_1,X_2|U=u) &= H\big{(}a_ub_u, (1-a_u)(1-b_u), (1-a_u)b_u, a_u(1-b_u)\big{)} \\
& = H\big{(}a_ub_u, (1-a_u)(1-b_u), a_u\star b_u\big{)} + (a_u\star b_u) h\left(\frac{a_u(1-b_u)}{a_u\star b_u}\right)
\end{align*}
Hence, noting that also $H(X_1,X_2|U=u) = h(a_u) + h(b_u)$ we obtain
\begin{align*}
  H(X_1+X_2|U=u) = F(a_u,b_u)
\end{align*}
where
\begin{align*}
F(y,z) &\triangleq  h(y) + h(z) - (y\star z) \cdot h\left(\frac{y(1-z)}{y\star z}\right)
\end{align*}

Our task is now reduced to upper bounding
\begin{align}\label{eq:condEnt}
\Expt F(a,b) = H(X_1+X_2|U)
\end{align}
subject to the constraints
\begin{align}
\nonumber  \Expt\,a &= \Pr(X_1=0) = \frac{1}{2} \\
\nonumber  \Expt\,b &= \Pr(X_2=0) = \frac{1}{2} \\
\nonumber  \Expt\,ab &= \Pr(X_1=0, X_2=0) = \frac{1}{2}(1-\eta)\\
\Expt h(a) &\geq r_1
 \label{eq:wyner_const}
\end{align}
In \cite{Wyner}, Wyner has upper bounded $\Expt h(a)+\Expt h(b)$ subject to the first three constraints. We extend his technique to account for the additional term and the additional entropy constraint.

The following proposition can be verified via standard analysis.
\begin{proposition}\label{prop:F}
  $F(y,z)$ is concave in the pair $(y,z)$. In addition $F(y,z) = F(z,y)$.
\end{proposition}
Define the random variable $\gamma = \tfrac{a+b}{2}$, and note that $\Expt \gamma = \frac{1}{2}$.
Using Proposition \ref{prop:F}, we have that
\begin{align}\label{eq:F1}
\nonumber    \Expt F(a,b)  &= \Expt \left(\frac{1}{2}F(a,b) + \frac{1}{2}F(b,a)\right) \\
\nonumber  &\leq \Expt F\left(\gamma,\gamma\right) \\
\nonumber   &= 2\Expt (h(\gamma) + \gamma^2 - \gamma) \\
& = -\frac{1}{2} + 2\Expt\left(h(\gamma) + (\gamma-\frac{1}{2})^2\right)
\end{align}
Defining $\theta = |\gamma-\tfrac{1}{2}|$ and letting $G(y) = h(\sqrt{y}+\tfrac{1}{2}) + y$ we have that
\begin{align}\label{eq:F2}
  \Expt F(a,b) \leq -\frac{1}{2} + 2\Expt G(\theta^2)
\end{align}
where we have used the symmetry of $h(\cdot)$ around $\tfrac{1}{2}$.

The following proposition can be verified via standard analysis.
\begin{proposition}
  $G(y)$ is concave and monotone decreasing over $[0,\frac{1}{4}]$.
\end{proposition}

Using~\eqref{eq:F2} and the concavity of $G(y)$ we obtain
\begin{align}\label{eq:G}
  \Expt F(a,b)  \leq -\frac{1}{2} + 2\Expt G(\theta^2) \leq -\frac{1}{2} + 2 G(\Expt \theta^2),
\end{align}
Since $G(y)$ is monotone decreasing, we can further upper bound~\eqref{eq:G} by replacing $\Expt\theta^2$ with any lower bound. To that end:
\begin{align}\label{eq:theta2}
\Expt (\theta^2) &= \Expt\left(\gamma-\frac{1}{2}\right)^2 = \Expt \gamma^2 - \Expt \gamma + \frac{1}{4} =  \frac{1}{4}(\Expt(a+b)^2  -1)
\end{align}
Hence, we need a lower bound on $\Expt(a+b)^2$, subject to the constraints~\eqref{eq:wyner_const}.
\begin{lemma}\label{lem:lambda}
Let $X,Y$ be two random variables satisfying $\Expt X^2<\infty$ and $\Expt XY = \mu \geq 0$. Assume further that $X\in\m{A}$ for some family $\m{A}$. Define
  \begin{align*}
    \lambda^* \triangleq \max \left\{\min_{X\in\m{A}} \frac{\mu}{\Expt X^2},\, 1\right\}.
  \end{align*}
Then
\begin{align*}
  \Expt(X+Y)^2 \geq \frac{(1+\lambda^*)^2}{\lambda^*}\mu
\end{align*}
\end{lemma}
\begin{proof}
  For any $X\in \m{A}$ define $\lambda_X \triangleq \frac{\mu}{\Expt X^2}$. For any $Y$ we can write
  \begin{align*}
    \Expt(X+Y)^2 &= \Expt (X+\lambda_XX + Y-\lambda_XX)^2 \\
&= (1+\lambda_X)^2\Expt X^2 + \Expt(Y-\lambda_XX)^2 + 2(1+\lambda_X)\Expt X(Y-\lambda_XX) \\
&\geq (1+\lambda_X)^2\Expt X^2
  \end{align*}
where the last inequality follows since $\Expt XY = \lambda_X\Expt X^2 = \mu $. Therefore,
\begin{align*}
  \Expt(X+Y)^2 \geq \frac{(1+\lambda_X)^2}{\lambda_X} \lambda_X\Expt X^2 = \frac{(1+\lambda_X)^2}{\lambda_X} \mu
\end{align*}
Note that the function $K(\lambda) = \frac{(1+\lambda)^2}{\lambda}$ has a unique minimum at $\lambda=1$. Define  $\lambda^\dagger \triangleq \min_{X\in\m{A}}\lambda_X$, and observe that $\lambda^* = \max\{\lambda^\dagger,1\}$. Hence if $\lambda^\dagger>1$ we can further lower bound the above by substituting $\lambda_X\to \lambda^\dagger$. Otherwise, we can replace $\lambda_X$ by $1$.
\end{proof}

We would  like to use Lemma \ref{lem:lambda} to lower bound $\Expt(a+b)^2$. To that end, define the zero mean random variables $\bar{a} = a-\tfrac{1}{2}$ and $\bar{b} = b-\tfrac{1}{2}$, and note that $\bar{a}$ must satisfy $h(\bar{a}+\frac{1}{2}) \geq r_1$. In order to apply the lemma we first need to upper bound $\Expt\bar{a}^2$ under this latter restriction.
\begin{lemma}\label{lem:h_var}
Let $X$ be a zero mean random variable over $[-\tfrac{1}{2},\tfrac{1}{2}]$ satisfying $\Expt h(X+\tfrac{1}{2})\geq \rho$. Then $\Expt X^2 \leq (\tfrac{1}{2}-h^{-1}(\rho))^2$, and this bound is tight.
\end{lemma}
\begin{proof}
  Define $Q(y) = h(\tfrac{1}{2} - \sqrt{y})$. It is easily verified that $Q(y)$ is concave over $[0,\tfrac{1}{4}]$. Then
  \begin{align*}
    \rho\leq \Expt h\left(\frac{1}{2} - X\right) = \Expt Q(X^2) \leq Q(\Expt X^2) = h\left(\frac{1}{2} - \sqrt{\Expt X^2}\right).
  \end{align*}
Using the monotonicity of $h^{-1}$ we get $\Expt X^2  \leq (\tfrac{1}{2} - h^{-1}(\rho))^2$. This bound is attained with equality by $X$ uniformly distributed over $\{\tfrac{1}{2} - h^{-1}(\rho), h^{-1}(\rho) - \tfrac{1}{2}\}$.
\end{proof}
Taking $\m{A}$ in Lemma \ref{lem:lambda} as the family of all random variables $\bar{a}$ distributed over $[-\tfrac{1}{2},\tfrac{1}{2}]$ with $h(\bar{a}+\tfrac{1}{2})\geq r_1$, and noting that $\Expt \bar{a}\bar{b} = \frac{1}{4}-\frac{1}{2}\eta\geq 0$, we can use Lemma \ref{lem:h_var} to express the associated $\lambda^*$ as
\begin{align*}
  \lambda^* &= \max\left\{\frac{\frac{1}{4}-\frac{1}{2}\eta}{\max_{\bar{a}\in\m{A}}\Expt \bar{a}^2},1\right\}  = \max\left\{\frac{\frac{1}{4}-\frac{1}{2}\eta}{(\tfrac{1}{2}-h^{-1}(r_1))^2},1\right\}\\
& = \max\left\{\frac{\frac{1}{2}-\eta}{\tfrac{1}{2}-h^{-1}(r_1)\star h^{-1}(r_1)},1\right\}
\end{align*}
and hence if $h^{-1}(r_1)\star h^{-1}(r_1) > \eta$ then
\begin{align*}
  \Expt(a+b)^2  &= 1+\Expt(\bar{a}+\bar{b})^2 \geq 1+ \frac{(1+\lambda^*)^2}{2\lambda^*}\left(\frac{1}{2}-\eta\right)\\
& = 1+ \frac{(1-\eta-h^{-1}(r_1)\star h^{-1}(r_1))^2}{1-2 (h^{-1}(r_1)\star h^{-1}(r_1))}
\end{align*}
and otherwise $\Expt(a+b)^2\geq 1+4\left(\tfrac{1}{4} - \tfrac{1}{2}\eta\right)$. Combining this with~\eqref{eq:condEnt},~\eqref{eq:G} and ~\eqref{eq:theta2} we obtain
\begin{align}\label{eq:entropy_bound}
H(X_1+X_2|U)& \leq -\frac{1}{2} + 2G\left(\frac{1}{4}(\Expt(a+b)^2-1)\right) \leq J(h^{-1}(r_1),\eta).
\end{align}

\begin{remark}
The above bound can be attained whenever $\eta\geq h^{-1}(r_1)\star h^{-1}(r_1)$. To show this, we specify a distribution that satisfies~\eqref{eq:wyner_const} (and therefore also $H(X_1+X_2) = h(\eta) + 1-\eta$), and satisfies the bound~\eqref{eq:entropy_bound} with equality. Let $p^*\leq \tfrac{1}{2}$ be such that $p^*\star p^* = \eta$, i.e., $p^* =  \frac{1}{2}(1-\sqrt{1-2\eta })$, and consider the following distribution:
\begin{align}\label{eq:opt_dist}
\nonumber &X_1 = U \oplus Z_1, \; X_2 = U \oplus Z_2 \\
&U\sim \textrm{Bern}\left(\frac{1}{2}\right),\; Z_1\sim\textrm{Bern}(p^*), \; Z_2\sim\textrm{Bern}(p^*)
\end{align}
where $U,Z_1,Z_2$ are mutually independent. Note that $X_2 = X_1 \oplus Z$, where $Z=(Z_1\oplus Z_2) \sim \textrm{Bern}(\eta)$. Hence, $\Expt X_1=\Expt X_2 = \tfrac{1}{2}$ and $\Pr(X_1=0,X_2=0) = \tfrac{1}{2}(1-\eta)$. Furthermore,
\begin{align*}
  H(X_1+X_2|U) &= \frac{1}{2}H(Z_1+Z_2) + \frac{1}{2}H(2-(Z_1+Z_2)) = H(Z_1+Z_2) \\
&= H(Z_1+Z_2,Z_1) - H(Z_1|Z_1+Z_2) = H(Z_1,Z_2) - H(Z_1|Z_1+Z_2) \\
& = 2h(p^*) - \eta\cdot H(Z_1|Z_1+Z_2=1) \\
&= 2h\left(\frac{1}{2}\left(1-\sqrt{1-2\eta}\right)\right)-\eta
\end{align*}
Since $\eta = p^*\star p^* \geq h^{-1}(r_1)\star h^{-1}(r_1)$, we also have that $H(X_1|U)=h(p^*) \geq r_1$. Therefore this distribution indeed satisfies the constraints. For $\eta<h^{-1}(r_1)\star h^{-1}(r_1)$, it is believed that the bound~\eqref{eq:entropy_bound} is not tight, due to the inequality in~\eqref{eq:F1}.
\end{remark}

Finally, we show that the constraints~\eqref{eq:wyner_const} cannot be satisfied if $\eta<h^{-1}(r_1)$.
\begin{lemma}\label{lem:CSI}
Let $X$ and $Y$ be two zero mean random variables on $[-\tfrac{1}{2},\tfrac{1}{2}]$. If $\Expt h(Y+\tfrac{1}{2}) \geq \rho$ then $\Expt XY \leq \tfrac{1}{2}(\tfrac{1}{2}-h^{-1}(\rho))$.
\end{lemma}
\begin{proof}
Clearly $\Expt X^2 \leq \tfrac{1}{4}$, and by Lemma \ref{lem:h_var} also $\Expt Y^2 \leq (\tfrac{1}{2}-h^{-1}(\rho))^2$. Using the Cauchy-Schwarz Inequality we have
\begin{align*}
(\Expt XY)^2 \leq \Expt X^2 \Expt Y^2 \leq \frac{1}{4}\left(\frac{1}{2}-h^{-1}(\rho)\right)^2
\end{align*}
\end{proof}
Using $\bar{a},\bar{b}$ in the above Lemma and recalling that $\Expt \bar{a}\bar{b} = \tfrac{1}{4}-\tfrac{1}{2}\eta$ and that $\Expt h(\bar{a}+\tfrac{1}{2}) \geq r_1$, we indeed verify that $\eta\geq h^{-1}(r_1)$ must hold.

The proof of Lemma \ref{lem:sw} now follows since we have shown that for any admissible $(r_0,r_1,r_2)$, there must exist an $h^{-1}(r_1)\leq\eta\leq \tfrac{1}{2}$ such that $r_0+r_1+r_2\leq L(\eta)$ and $r_1+r_2\leq J(p,\eta)$.

\section{Discussion}
Given a pair of multiset-union-free families $\m{F}_1,\m{F}_2$ of subsets of $[n]$ with cardinalities $2^{nR_1}$ and $2^{nR_2}$ respectively, we have introduced a bounding technique based on a procedure for constructing a  multiset-union-free system $\m{U}$ over subsets of $[(1-\alpha) n]$, for $\alpha<1$. This was achieved by proving the existence of a subset $S\subset [n]$ of cardinality $\alpha n$, such that the multiset-union of the projection multisets of each family on $S$, i.e., $P^+_S(\m{F}_1)\uplus P^+_S(\m{F}_2)$ has a member $T$ with a large number of multiplicities, say $2^{n\rho}$. This in turn implied that $r_0+r_1+r_2$ for the system is at least $\rho/(1-\alpha)$. To lower bound $\rho$ as a function of $\alpha$ and the cardinalities of the original families, we introduced the soft Sauer-Perles-Shelah Lemma, which enabled us to bound the number of occurrences of the multiset $T=S$. This lemma offered the additional benefit of a lower bound on $r_1$. We note in passing that the bound obtained on $R_2$ as a function of $R_1$ outperforms previous results even without incorporating the constraint on $r_1$. We suspect that better bounds on $\rho$ can be obtained, possibly for $T$ other than $S$.

\end{document}